\documentclass[12pt, twoside]{article}
\usepackage{amsfonts}
\usepackage{amsmath,amssymb}
\usepackage{multicol}
\usepackage{color}

\newtheorem{theorem}{Theorem}[section]
\newtheorem{lemma}{Lemma}[section]

\newtheorem{remark}{Remark}[section]

\numberwithin{equation}{section}
\newenvironment{proof}[1][Proof]{\noindent\textbf{#1.} }{\hfill $\Box$}
\allowdisplaybreaks

 \makeatletter
\begin{document}
\author{ {\small Long-Jiang Gu$^a$,  Xiaoyu Zeng$^b$ and  Huan-Song Zhou$^{b}$}\thanks{Corresponding
author. Email: hszhou2002@sina.com
%; Tel: +86-18171231798
. To appear in Nonlinear Analysis. DOI: 10.1016/j.na.2016.10.002}
%\thanks{This work was supported by the NFSC No.11471331 and No.11501555.}
 \\
 %\thanks{ Supported by Program for New
% Century Excellent Talents in University and by  FRFCU(lzujbky-2011-43).}
% \\
%EndAName
{\small  $^a$Wuhan Institute of Physics and Mathematics, Chinese Academy of Sciences,}\\
{\small  P.O. Box 71010, Wuhan 430071, PR China}\\
{\small $^b$Department of Mathematics, Wuhan University of Technology, Wuhan 430070, China}
}
\title{Eigenvalue problem for a  p-Laplacian equation with trapping potentials
  }
  \date{}
  \maketitle

\begin{abstract}
Consider the following eigenvalue problem of p-Laplacian equation
\begin{equation}\label{P}
    -\Delta_{p}u+V(x)|u|^{p-2}u=\mu|u|^{p-2}u+a| u|^{s-2}u, x\in \mathbb{R}^{n},  \tag{P}
\end{equation}
where $a\geq0$, $p\in (1,n)$ and $\mu\in\mathbb{R}$. $V(x)$ is a trapping type potential, e.g., $\inf\limits_{x \in \mathbb{R}^n}V(x)< \lim\limits_{|x|\rightarrow+\infty }V(x)$. By using constrained variational methods, we proved that there is $a^*>0$, which can be given explicitly, such that problem (\ref{P}) has a ground state $u$  with $\|u\|_{L^p}=1$ for some $\mu \in \mathbb{R}$ and all $a\in [0,a^*)$, but (\ref{P}) has no this kind of  ground state
% with normalized $L^p$-norm
 if $a\geq a^*$.
Furthermore, by establishing some delicate energy estimates we show   that the global maximum point of the ground states of problem (\ref{P}) approach to one of the global minima of $V(x)$ and blow up if $a\nearrow a^*$. The optimal rate of blowup is obtained for $V(x)$ being a polynomial type potential.
\end{abstract}

\section{Introduction}
\noindent

In this paper, we are concerned with the existence and asymptotical behavior of ground states for the following eigenvalue problem of p-Laplacian
equation:
 \begin{equation}\label{Equation}
    -\Delta_{p}u+V(x)|u|^{p-2}u=\mu|u|^{p-2}u+a| u|^{s-2}u, x\in \mathbb{R}^{n},
\end{equation}
where $p\in (1,n)$, $s=p+\frac{p^2}{n}$, $a\geq0$ and $\mu\in\mathbb{R}$ are parameters, $V(x)$ is a trapping potential which satisfies
\begin{itemize}
  \item[]  $(V)$ :  $V(x)\in C(\mathbb{R}^n)$, $\lim \limits_{|x|\rightarrow \infty}V(x)=\infty$ and $\inf\limits_{x\in\mathbb{R}^n}V(x)=0$.
\end{itemize}

When $p=n=2$, (\ref{Equation}) is the so called time independent  Gross-Pitaevskii (GP) equation, which was proposed independently by Gross\cite{G}
and Pitaevskii\cite{P} in studying the Bose-Einstein condensate. In this special case, problem (\ref{Equation}) has been studied under various
conditions on the potential $V(x)$, for examples, \cite{guo,zeng,wang}, etc. Roughly speaking, if $V(x)$  is a polynomial type trapping potential such as
\begin{equation}\label{polynomial V}
   V(x)=h(x)\prod \limits_{i=1}^{m} |x-x_{i}|^{q_{i}}, x_i\neq x_j   \text{ if } i\neq j, 0<C\leq h(x)\leq \frac{1}{C}   \text{ for all }  x\in \mathbb{R}^{n},
\end{equation}
the results of \cite{guo} show that the existence of normalized $L^2$-norm ground states of (\ref{Equation}) depends heavily on the parameter
$a\geq 0$, and this kind of solution blows up at some point $x_{i_{0}}$ with $q_{i_{0}}=max \{q_1,...,q_m\}$. The rate of blowup is also given in \cite{guo}.
The main aim of this paper is to extend the results of \cite{guo} to the p-Laplacian problem (\ref{Equation}) for general $p\in(1,n)$ and $V(x)$.

As we know, the operator $-\Delta_{p}$ is no more linear if $p\neq 2$, which leads to some quite different properties from $-\Delta$ (i.e. $p=2$),
for examples, it is well known that the limit equation of (\ref{Equation}), that is,
\begin{equation}\label{limit equation}
     -\Delta_{p}u+\frac{p}{n}|u|^{p-2}u=| u|^{s-2}u,
\end{equation}
has a unique positive radially symmetric solution (see e.g., \cite{GNN,K,Li,mcleod}) for $p=2$, but in general case we know that this fact holds only for $p\in (1,2)$
(see e.g., \cite{damascelli,gong,tang}), which is still unknown if
$p\in(1,n)$ and $p\neq 2$. However, in \cite{guo} the uniqueness of solutions  of the  limit equation of (\ref{Equation}) plays a crucial role not only
in applying the  Gagliardo-Nirenberg inequality, but also in getting the exact blowup rate for the ground state of (\ref{Equation}). In order to extend
the results of \cite{guo} to the p-Laplacian case, the key step is how to avoid using the uniqueness of solutions of the limit equation  (\ref{limit equation}).
In this paper, we overcome this difficulty by detailed analyzing the  relations between the extremal functions of the sharp constant of the Gagliardo-Nirenberg inequality
and the ground states of the limit equation  (\ref{limit equation}).
On the other hand, if $p\neq2$, the expansion of the main part of the variational functional of (\ref{Equation}) is also more complicated than that of $p=2$,
which causes more difficulties in making the energy estimates than in \cite{guo}.

To get a ground state solution of (\ref{Equation}), we consider the following constrained minimization problem
\begin{equation}\label{minimizing problem}
 e(a)=\inf\{E_a(u): u\in\mathcal{H}, \int_{\mathbb{R}^n}|u|^p dx=1\},
\end{equation}
where $E_a$ is the energy functional defined by
\begin{eqnarray}
\nonumber  E_{a}(u) &=& \int_{\mathbb{R}^n} \left(|\nabla u|^p+V(x)|u|^p
\right)dx-\frac{pa}{s}\int_{{\mathbb{R}}^n} |u|^s dx \\
   &=&   \int_{\mathbb{R}^n} \left(|\nabla u|^p+V(x)|u|^p\right)dx-\frac{na}{n+p}\int_{{\mathbb{R}}^n} |u|^s dx,\label{GP}
\end{eqnarray}
for $u\in \mathcal{H}:=\Big\{ {u\in W^{1,p}(\mathbb{R}^{n}): \int_{\mathbb{R}^n}
	V(x)|u|^p dx<+\infty} \Big\}$ and
\begin{equation*}
\|u\|_{\mathcal{H}}:={ \left(\int_{\mathbb{R}^n}
|\nabla u|^p+V(x)|u(x)|^p dx \right)}^{\frac{1}{p}}.
\end{equation*}
Clearly, a minimizer of (\ref{minimizing problem}) is a weak solution of (\ref{Equation}) for some $\mu\in\mathbb{R}$, which is indeed a
Lagrange multiplier.

For problem (\ref{minimizing problem}), the power $s=p+\frac{p^2}{n}$ is critical in the sense that $e(a)$ can be $-\infty$ if $s>p+\frac{p^2}{n}$ .
Indeed , take $u\in \mathcal{H}$ and $\|u\|_{L^p(\mathbb{R}^n)}=1$ and let $v_{\lambda}(x)=\lambda^{\frac{n}{p}}u(\lambda x)$, it is easy to see that
\begin{equation*}
    E_a(v_{\lambda})=\lambda^{p}\int_{\mathbb{R}^n}  |\nabla u|^p dx +\int_{{\mathbb{R}}^n}V(\frac{x}{\lambda})|u|^p dx-\frac{pa}{s}\lambda^{\frac{ns}{p}-n}
\int_{{\mathbb{R}}^n} |u|^s dx \overset{\lambda\rightarrow\infty}{\longrightarrow}-\infty,
\end{equation*}
if $\frac{ns}{p}-n>p$, i.e.,  $s>p+\frac{p^2}{n}$. When
$p=n=2$, $s=4$ is the so called mass critical exponent for GP equation.

For $s=p+\frac{p^2}{n}$, we recall some known results about the limit equation
(\ref{limit equation}). First we define the energy functional
\begin{equation*}\label{limit energy functional}
    I(u)=\frac{1}{p}\int_{\mathbb{R}^n} ( |\nabla u|^p +\frac{p}{n}|u|^p) dx-\frac{1}{s}\int_{\mathbb{R}^n}|u|^s dx, u\in W^{1,p}(\mathbb{R}^n).
\end{equation*}
It is well known that $u$ is a weak solution of (\ref{limit equation}) if and only if
$$ \langle I'(u),\varphi\rangle =0 \text{ for all } \varphi \in W^{1,p}(\mathbb{R}^n).$$
Next, we denote the set of all nontrivial weak solutions of (\ref{limit equation}) by $\mathcal{S}$, that is
$$\mathcal{S}:=\{u\in W^{1,p}(\mathbb{R}^n)\setminus\{0\}: \langle I'(u),\varphi\rangle =0 \quad \forall \varphi \in W^{1,p}(\mathbb{R}^n)\}.$$
Then, for any $u\in \mathcal{S}$,  by using Pohozaev identity(see \cite{Poho}) we know that
\begin{equation}\label{Q}
    \int_{\mathbb{R}^{n}} |u|^s dx= (1+\frac{p}{n})\int_{\mathbb{R}^{n}} |\nabla u|^p dx= (1+\frac{p}{n})\int_{\mathbb{R}^{n}} | u|^p dx,
\end{equation}
where $s=p+\frac{p^2}{n}$. Now, we say $Q\in W^{1,p}(\mathbb{R}^n)$ is a {\it ground state}  of (\ref{limit equation}) if it is the {\it least energy solution}
among all nontrivial  weak solutions of (\ref{limit equation}). Then, it follows from (\ref{Q}) that
\begin{align}
% to remove numbering (before each equation)
Q\in \mathcal{G}:=  \{u\in \mathcal{S}: I(u)= \inf\limits_{v\in \mathcal{S}}I(v)\} %\nonumber \\
    = \{u\in \mathcal{S}: I(u)= \inf\limits_{v\in \mathcal{S}}\frac{1}{n}\int_{\mathbb{R}^{n}} |v|^p dx\}. \label{jitai}
\end{align}
Clearly, if $Q(x) \in \mathcal{G}$, then $Q(x-x_0)\in \mathcal{G}$ for any $x_0 \in \mathbb{R}^n$. Furthermore, by the result of \cite{gong}, any ground state of (\ref{Equation}) decays exponentially at infinity, that is,
for any $Q\in \mathcal{G}$ there exists $\delta>0$  such that
\begin{equation}\label{decay}
    |Q(x)|\leq e^{-\delta|x|}, \text{ for } |x| \text{ large. }
\end{equation}

Finally, we give the main theorems of the paper. Our first theorem is concerned with the existence of minimizers of  the minimization problem
(\ref{minimizing problem}) and hence Lemma \ref{lemma 2.4} implies the existence of ground states of (\ref{Equation}), which is consistent with the results of \cite{Bao Cai YY,guo,Zhangjian}
if $p=2$.
\begin{theorem}\label{mainresult}
Let $Q\in\mathcal{G}$ and let
\begin{equation}\label{a*}
  a^*= \Big( \int_{\mathbb{R}^{n}} |Q|^p \Big)^\frac{p}{n}.
  \end{equation}
If  $ p\in (1,n)$ and  $V(x)$ satisfies the  condition $(V)$.  Then,
\begin{description}
  \item[(i)] Problem (\ref{minimizing problem}) has at least one minimizer if $0\leq a<a^*$.
  \item[(ii)] Problem (\ref{minimizing problem}) has no minimizer if $a\geq a^*$ and $e(a)=-\infty$ if $a>a^{*}$.
Moreover, $e(a)>0$ if $a<a^*$ and $\lim\limits_{a\nearrow a^*}e(a)=e(a^*)=0$.
  \end{description}
\end{theorem}

\begin{remark}
The number $a^*$ defined in (\ref{a*}) is independent of the choice of $Q \in \mathcal{G}$. In fact, let $c_0$ be the least energy of (\ref{limit equation}),
then, for any $Q\in \mathcal{G}$, $I(Q)=c_0$ and it follows from (\ref{jitai}) that $ \int_{\mathbb{R}^{n}} |Q|^p dx=n c_0$, which is independent of $Q \in \mathcal{G}$.
\end{remark}

By  Theorem \ref{mainresult}, we know that, for any $a\in [0,a^*)$, problem (\ref{minimizing problem}) has a solution $u_a$, then it is interesting to ask
what would happen when $a$ goes to $a^*$ from below, which is simply denoted by $a\nearrow a^*$ in what follows. Our next theorem answers this question
for the general type of trapping potential $V(x)$ as in (V).

\begin{theorem}\label{mainresult2}
Let $u_{a}\geq 0$ be a  minimizer of (\ref{minimizing problem}) for $a\in (0,a^{*})$. If the condition $(V)$ holds, then

\begin{description}
  \item[(i)] \begin{equation}\label{epxl}
    \varepsilon_{a}\triangleq\big(\int_{\mathbb{R}^{n}}|\nabla u_{a}|^{p}\big)^{-\frac{1}{p}}\rightarrow 0 \quad as \quad a\nearrow a^{*}.
    \end{equation}
  \item[(ii)] Let $\bar{z}_{a}$ be a global  maximum point of $u_{a}(x)$, there holds

\begin{equation}\label{dist}
\lim \limits_{a\nearrow a^{*}}  dist (\bar{z}_{a},\mathcal{A})=0,
\end{equation}
where $\mathcal{A}=\{x \in \mathbb{R}^{n}:V(x)=0\}$.

   \item [(iii)]For any sequence $\{a_{k}\}$ with $a_{k}\nearrow a^{*}$ as
 $k\rightarrow\infty$, there exists a subsequence of $\{a_{k}\}$, still denoted by $\{a_{k}\}$, such that
 \begin{equation}\label{jixian}
 \lim \limits_{k\rightarrow\infty}\varepsilon_{a_k}^{\frac{n}{p}}u_{a_{k}}(\varepsilon_{a_k}x+\bar{z}_{a_k})=\frac{Q(x)}{a^{*\frac{n}{p^{2}}}}
\text{ in } W^{1,p}(\mathbb{R}^n), \text{ for some } Q\in \mathcal{G},
 \end{equation}
 where  $\bar{z}_{a_k}$ is a global maximum point of $u_{a_k}$  and $\lim \limits_{k\rightarrow\infty}\bar{z}_{a_k}
 =x_{0} \in \mathcal{A}$.
% \begin{equation}\label{varepsilon}
%    \varepsilon_{k}:=(\int_{\mathbb{R}^{n}}|\nabla u_{a_{k}}|^{p})^{-\frac{1}{p}}\rightarrow 0 \quad as \quad k\rightarrow\infty.
% \end{equation}
\end{description}

 \end{theorem}

The above theorem tells us that as $a\nearrow a^*$, the minimizers of (\ref{minimizing problem}) must concentrate and blow up at a
minimum point of $V(x)$. Our final result shows that the concentration behavior and blow up rate of the minimizers of (\ref{minimizing problem}) can be
refined if we have more information on the potential $V(x)$.
Accurately, we assume that the trapping potential $V(x)$ is of some ``polynomial type'' and has $m\geq1$ isolated minima, for instance,
$V(x)$ is given by (\ref{polynomial V}).
Let     $Q \in \mathcal{G}$  be   given in Theorem \ref{mainresult2}, and let
 $y_0 \in \mathbb{R}^n$  be such that
\begin{equation}\label{y_0}
        \int_{\mathbb{R}^{n}} |x+y_0|^{q}Q^{p}(x)dx= \inf \limits_{y\in \mathbb{R}^n} \int_{\mathbb{R}^{n}} |x+y|^{q}Q^{p}(x)dx, \quad q=max\{q_{1},...,q_{m}\}.
\end{equation}
Set
\begin{equation*}
    \lambda_{i}=\int_{\mathbb{R}^{n}} |x+y_0|^{q}Q^{p}(x)dx \lim \limits_{x\rightarrow x_{i}} \frac{V(x)}{|x-x_{i}|^{q}} \in (0,\infty],
\end{equation*}
 and
 \begin{equation}\label{lambda}
    \lambda=min\{\lambda_{1},...,\lambda_{m}\}, \quad \mathcal{Z}:=\{x_{i}:\lambda_{i}=\lambda\}.
 \end{equation}

\begin{remark}
If $1<p\leq 2$,   the  ground state   $Q$ of (\ref{limit equation}) is unique (up to translation) and
radially symmetric, see e.g., \cite{tang,damascelli,GNN,K,Li,mcleod}, then it is not difficult to know that $y_0=0$  in (\ref{y_0}).
\end{remark}

Based on  Theorem \ref{mainresult2} and the above notations, we have the following theorem, which is a refined version of Theorem \ref{mainresult2}
when the potential $V(x)$ is given by (\ref{polynomial V}).

\begin{theorem}\label{mainresult3}
If $V(x)$ satisfies (\ref{polynomial V}), let $\{a_k\} \subset (0,a^*)$ be the convergent subsequence in  Theorem \ref{mainresult2} $\mathbf{(iii)}$ and let $u_{a_{k}}$ be a   corresponding minimizer of (\ref{minimizing problem}), then
\begin{description}
  \item[(i)] For $e(a_k)$ defined by (\ref{minimizing problem}), there holds
\begin{equation}\label{energy}
    e(a_k)\approx \frac{(a^*-a_k)^{\frac{q}{p+q}}}{{a^{*}}^{\frac{n+q}{p+q}}}\lambda^{\frac{p}{p+q}} ((\frac{q}{p})^{\frac{p}{p+q}}+(\frac{p}{q})^{\frac{q}{p+q}}) \quad as \quad k\rightarrow \infty,
\end{equation}
where $f(a_k)\approx g(a_k)$ means that $f/g\rightarrow 1$ as $k\rightarrow \infty$.
  \item[(ii)] Let $Q \in \mathcal{G}$ be obtained in (\ref{jixian})
  and let
  $\varepsilon_{a_k}$ be defined by (\ref{epxl}), then (\ref{jixian}) still holds for $u_{a_{k}}$, but $\varepsilon_{a_k}$ can be precisely estimated as
  \begin{equation} \label{sigma}
  \varepsilon_{a_k} \approx \sigma_{k} \triangleq {a^{*}}^{\frac{n-p}{p(p+q)}}
 (a^{*}-a_{k})^{\frac{1}{p+q}} \lambda^{-\frac{1}{p+q}}(\frac{p}{q})^{\frac{1}{p+q}},
  \end{equation}
  that is,
 \begin{equation}\label{jixian2}
 \lim \limits_{k\rightarrow\infty}\sigma_{k}^{\frac{n}{p}}u_{a_{k}}(\sigma_{k}x+\bar{z}_{a_k})=\frac{Q(x)}{a^{*\frac{n}{p^{2}}}}
 \text{ in } W^{1,p}(\mathbb{R}^n).
\end{equation}
  %Define $\sigma_{k}={a^{*}}^{\frac{n-p}{p(p+q)}}(a^{*}-a_{k})^{\frac{1}{p+q}} \lambda^{-\frac{1}{p+q}}(\frac{p}{q})^{\frac{1}{p+q}}$,  then
 %\begin{equation}\label{sigma}
 %   \lim \limits_{k\rightarrow\infty} \frac{\sigma_k}{\varepsilon_k}=1.
% \end{equation}
 Moreover, for each $k$, if $\bar{z}_{a_k}$  is a global maximum point of $u_{a_{k}}$, then
  \[
  \lim \limits_{k\rightarrow\infty}\bar{z}_{a_k}=x_{0} \text{ with } x_{0}\in \mathcal{Z}.
  \]
  %  and
 %\begin{equation}\label{jixian2}
% \lim \limits_{k\rightarrow\infty}\sigma_{k}^{\frac{n}{p}}u_{a_{k}}(\sigma_{k}x+\bar{z}_{a_k})=\frac{Q(x)}{a^{*\frac{n}{p^{2}}}}
 %\text{ in } W^{1,p}(\mathbb{R}^n).
%\end{equation}

\end{description}
 \end{theorem}

\section{Preliminary lemmas}
\noindent

In this section, we give some useful lemmas which are required in next section.

\begin{lemma}\label{Gagliardo-Nirenberg inequality}
$(\text{Gagliardo-Nirenberg inequality})$ Let $p\in (1,n)$, $s=p+\frac{p^2}{n}$ and $a^*$ be given by  (\ref{a*}). Then, for any
$u\in W^{1,p}(\mathbb{R}^{n})$, there holds

\begin{equation}\label{GNinequality}
\int_{\mathbb{R}^{n}} |u(x)|^s dx  \leq \frac{n+p}{n a^*}
\int_{\mathbb{R}^{n}}|\nabla u(x)|^p dx \cdot \Big(\int_{\mathbb{R}^{n}}|u(x)|^p dx\Big)^{\frac{p}{n}}.
\end{equation}
Moreover, the equality holds if and only if $u(x)=c_1 Q(c_2 x)$ for some $c_1,c_2 \in \mathbb{R}\setminus\{0\}$ and $Q\in \mathcal{G}$.
\end{lemma}

\begin{proof}
By using Theorem 2.1 of \cite{Sharp GN} with $q=p $ and  $s=p+\frac{p^2}{n}$, we see that
\begin{equation}\label{lemma2.1 1}
\int_{\mathbb{R}^{n}} |u(x)|^s \leq \Big(  \frac{K}{E(u_{\infty})}  \Big)^{\frac{n+p}{n}}
\int_{\mathbb{R}^{n}}|\nabla u(x)|^p\cdot \Big(\int_{\mathbb{R}^{n}}|u(x)|^p\Big)^{\frac{p}{n}},
\end{equation}
where $K=\frac{n+p}{np(\frac{p}{n})^{\frac{p}{n+p}}}$ and $u_{\infty}$ is a minimizer of the following constrained minimization problem:
\begin{equation}\label{lemma2.1 2}
    \inf \Big\{E(u)=  \frac{1}{p}\int_{\mathbb{R}^{n}}|\nabla u(x)|^p dx + \frac{1}{p} \int_{\mathbb{R}^{n}}|u(x)|^p  dx :u\in W^{1,p}(\mathbb{R}^{n})
    , \|u\|_{L^s(\mathbb{R}^{n})}=1 \Big\},
\end{equation}
%%and $u=c_1 u_{\infty} (c_2 x+\bar{x})$ are optimal functions in inequality (\ref{lemma2.1 1}) for arbitrary $c_1,c_2\neq 0$ and $\bar{x} \in \mathbb{R}^n$.

Since $u_\infty$ is a minimizer of (\ref{lemma2.1 2}), then $u_\infty$ satisfies
\begin{equation*}\label{lemma2.1 3}
     -\Delta_{p}u_\infty+|u_\infty|^{p-2}u_\infty= \lambda |u_\infty|^{s-2}u_\infty,
\end{equation*}
where
\begin{equation}\label{lemma2.1 4}
    \lambda=p E(u_\infty),
\end{equation}
is the so called Lagrange multiplier. By the Pohozaev identity\cite{Poho}, we  have
\begin{equation}\label{lemma2.1 5}
    \int_{\mathbb{R}^{n}} |u_\infty|^p  dx=\frac{p}{n}\int_{\mathbb{R}^{n}}|\nabla u_\infty|^p dx=\frac{\lambda p}{n+p}\int_{\mathbb{R}^{n}} |u_\infty|^s  dx.
\end{equation}
Let
\begin{equation}\label{lemma2.1 6}
    v(x)=(\frac{p \lambda}{n})^{\frac{n}{p^2}} u_{\infty} \big(    (\frac{p}{n})^{\frac{1}{p}}        x \big),
\end{equation}
then $v$ satisfies  (\ref{limit equation}). By the definition of ground state and (\ref{a*}), it follows from (\ref{lemma2.1 6}) that
\begin{equation}\label{lemma2.1 8}
  \lambda^{\frac{n}{p}}\int_{\mathbb{R}^{n}} |u_{\infty}|^p  dx= \int_{\mathbb{R}^{n}} |v|^p  dx \geq {a^*}^{\frac{n}{p}}.
\end{equation}
Hence, (\ref{lemma2.1 5}) and (\ref{lemma2.1 8}) together with   $\int_{\mathbb{R}^{n}} |u_{\infty}|^s =1 $ imply that
\begin{equation}\label{lemma2.1 9}
   \lambda \geq (\frac{n+p}{p})^{\frac{p}{n+p}} {a^*}^{\frac{n}{n+p}}.
\end{equation}
So,  (\ref{GNinequality}) holds by using (\ref{lemma2.1 1}), (\ref{lemma2.1 4}) and (\ref{lemma2.1 9}).

Next, we claim that any $Q\in \mathcal{G}$ is an extremal function  of (\ref{GNinequality}).

Indeed, if $Q\in \mathcal{G}$ then $Q$ is a ground state  of (\ref{limit equation}),  and   (\ref{Q}) and (\ref{a*}) hold, that is,
\begin{equation*}
  \frac{n} {n+p} \int_{\mathbb{R}^{n}} |Q|^s dx=    \int_{\mathbb{R}^{n}} |\nabla Q|^p dx=   \int_{\mathbb{R}^{n}} | Q|^p dx=
   {a^*}^{\frac{n}{p}}.
\end{equation*}
Therefore,   $Q$ satisfies the equality   of (\ref{GNinequality}), so does $c_1 Q(c_2 x)$ for any $c_1, c_2 \in \mathbb{R}\backslash\{0\}$.

%So far we have proved a scaling of any elements in $\mathcal{G}$ is an extremal of (\ref{GNinequality}), at last we will prove that any extremal of (\ref{GNinequality}) must equal to a scaling of some elements in $\mathcal{G}$.Indeed  suppose

Now, let $u$ be an extremal of (\ref{GNinequality}), then by a similar arguments to  \cite{weinstein} we know that $u$ satisfies
\begin{equation*}
    -\Delta_{p}u+a_{1}|u|^{p-2}u -a_{2}|u|^{s-2}u=0,
\end{equation*}
for some $a_{1},a_{2}>0$. Let $\hat{u}(x)=(\frac{pa_2}{na_1})^{\frac{n}{p^2}}u\big( (\frac{p}{na_1})^{\frac{1}{p}} x  \big)$,  then  $\hat{u}$ satisfies
 \begin{equation*}
-\Delta_{p}\hat{u}+\frac{p}{n}|\hat{u}|^{p-2}\hat{u}-|\hat{u}|^{s-2}\hat{u}=0,
\end{equation*}
and $\hat{u}$ is also an extremal of (\ref{GNinequality}). Then by   (\ref{Q}) we have
\begin{equation*}
     \int_{\mathbb{R}^{n}} |\hat{u}|^p =(a^*)^{\frac{n}{p}}.
\end{equation*}
Thus, $\hat{u} \in \mathcal{G} $ and $u=(\frac{na_1}{pa_2})^{\frac{n}{p^2}}\hat{u}\big( (\frac{na_1}{p})^{\frac{1}{p}} x  \big) $.

\end{proof}

Using Lemma \ref{Gagliardo-Nirenberg inequality} one can quickly get the following result.

\begin{lemma}\label{lemma 2.2}
Let $0\leq w_{0}\in W^{1,p}(\mathbb{R}^{n})$ satisfy the equation
\begin{equation}\label{lemma 2.2.1}
     -\Delta_{p} w_0+\frac{p}{n}{w_0}^{p-1}=a^{*} {w_0}^{s-1},
\end{equation}
and
\begin{equation}\label{lemma 2.2.2}
 \int_{\mathbb{R}^{n}} |\nabla w_0|^p dx= \int_{\mathbb{R}^{n}} | w_0|^p dx=1.
\end{equation}
Then, $w_0$ satisfies the equality of (\ref{GNinequality})  and  $w_0={a^{*}}^{-\frac{n}{p^2}}Q(x)$ for some
$Q\in \mathcal{G}$.
\end{lemma}

\begin{proof}
 By (\ref{lemma 2.2.1}) and (\ref{lemma 2.2.2}), it is easy to see that
\begin{equation}\label{lemma 2.2.3}
    \int_{\mathbb{R}^{n}} | w_0|^s dx=\frac{n+p}{na^*},
\end{equation}
 using again (\ref{lemma 2.2.2}) we konw that $w_0$ satisfies the equality of (\ref{GNinequality}). Then,  Lemma \ref{Gagliardo-Nirenberg inequality} implies  that there exist $c_1,c_2>0$  such that
\begin{equation*}
    w_0=c_1 Q(c_2 x ),
\end{equation*}
for some $Q\in \mathcal{G}$.  Using (\ref{lemma 2.2.2}) and (\ref{lemma 2.2.3}) together with (\ref{Q}) and
 (\ref{a*}), we get  $c_1={a^{*}}^{-\frac{n}{p^2}} $ and $c_2=1 $, thus the proof is completed.
\end{proof}

\begin{lemma}\label{compactlemma}
	Suppose $V\in L^{\infty}_{\text{loc}}(\mathbb{R}^n)$
	with $\lim\limits_{|x|\rightarrow \infty}V(x)=\infty$,
	then the embedding $\mathcal{H} \hookrightarrow L^q( \mathbb{R}^n )$
	is compact, for any $p\leq q <p^{*}=\left\{
   \begin{array}{ll}
     \frac{np}{n-p},  p<n,\\
     +\infty , p\geq n.
   \end{array}
   \right.$
\end{lemma}

\begin{proof}
This lemma can be proved by almost the same way as  that of Lemma 5.1  in  \cite{Zhangjian}  or section 3 of \cite{wangzhiqiang}, where only $p=2$ is considered.
\end{proof}

\begin{lemma}\label{lemma 2.4}
If $u_a$ is a minimizer of problem (\ref{minimizing problem}), then $u_a$ is a ground state  of (\ref{Equation}) for
some $\mu= \mu_a$.
\end{lemma}
\begin{proof}
Let $u_a$ be a minimizer of problem (\ref{minimizing problem}), then there is a $\mu_a$, i.e., the so called Lagrange multiplier, such that
\begin{equation}\label{lemma2.4 EL}
       -\Delta_{p}u_{a}+V(x)|u_{a}|^{p-2}u_{a}=\mu_{a}|u_{a}|^{p-2}u_{a}+a|u_{a}|^{s-2}u_{a}.
\end{equation}
Define
\begin{equation}\label{lemma2.4 J}
    J_a(u)=\frac{1}{p} \int_{\mathbb{R}^n} \big[ |\nabla u|^p +(V(x)-\mu_a)|u|^p \big] dx-\frac{a}{s}\int_{\mathbb{R}^n}|u|^s dx.
\end{equation}
Then, to prove $u_a$ is a ground state of (\ref{Equation}), we need only to show that $J_a(u_a)\leq J_a(v)$ for any nontrivial weak solution
$v$  of (\ref{lemma2.4 EL}). For this purpose, let $v(x)\not\equiv 0$ be solution of  (\ref{lemma2.4 EL}), then we see that
\begin{equation}\label{lemma2.4 *}
    \int_{\mathbb{R}^n}  \big[ |\nabla v|^p +(V(x)-\mu_a)|v|^p \big] dx = a \int_{\mathbb{R}^n}|v|^s dx.
\end{equation}
It follows from (\ref{lemma2.4 J}) and (\ref{lemma2.4 *}) that
\begin{equation*}\label{lemma2.4 sanjiao}
     J_a(v)=(\frac{1}{p}-\frac{1}{s})a\int_{\mathbb{R}^n}|v|^s dx=\frac{a}{n+p}\int_{\mathbb{R}^n}|v|^s dx.
\end{equation*}
Since $u_a$ satisfies (\ref{lemma2.4 EL}), then (\ref{lemma2.4 *}) holds also for $u_a$, this implies  that
\begin{equation*}\label{lemma2.4 **}
    J_a(u_a)=\frac{a}{n+p}\int_{\mathbb{R}^n}|u_a|^s dx.
\end{equation*}
Now, we set $d=\|v\|_{L^{p}(\mathbb{R}^n)}$ and $\bar{v}=\frac{v(x)}{d}$, then $\|\bar{v}\|_{L^{p}(\mathbb{R}^n)}=1$. Note that
$u_a$ is a minimizer of (\ref{minimizing problem}), hence
\begin{equation*}
    E_{a}(\bar{v}) \geq  E_{a}(u_a), \|u_a\|_{L^{p}(\mathbb{R}^n)}=1,
\end{equation*}
which means that
\begin{equation}\label{lemma2.4 wujiaoxing}
    J_a(\bar{v})=\frac{1}{p}E_{a}(\bar{v})-\frac{\mu_a}{p} \int_{\mathbb{R}^n}|\bar{v}|^p dx \geq \frac{1}{p}E_{a}(u_a)-\frac{\mu_a}{p} \int_{\mathbb{R}^n}|u_a|^p dx=J_a(u_a).
\end{equation}
On the other hand, by the definition of $\bar{v}$ and $J_{a}$ as well as (\ref{lemma2.4 *}), we see that
\begin{equation*}
    J_{a}(\bar{v})=\big( \frac{1}{pd^p}-\frac{1}{sd^{s}}  \big)a \int_{\mathbb{R}^n}|v|^s dx\leq \frac{a}{n+p} \int_{\mathbb{R}^n}|v|^s dx=J_{a}(v),
\end{equation*}
this and (\ref{lemma2.4 wujiaoxing}) show that $J_{a}(v)\geq J_{a}(u_a)$. We complete the proof.

\end{proof}

\begin{lemma}\label{a b Cp}
Let $a>0$, $b>0$ and $p>1$. then there exists $C_p=C(p)>0$ such that
\begin{equation*}
    (a+b)^p \leq a^p +b^p+C_p a^{p-1}b+C_p ab^{p-1}.
\end{equation*}
\end{lemma}
\begin{proof}[]
\end{proof}

\section{Existence of ground states}
\noindent

The aim of this section is to prove Theorem \ref{mainresult}. By Theorem \ref{mainresult} and Lemma \ref{lemma 2.4} we then
 get the existence and non-existence of ground states of (\ref{Equation}). For the $p=2$ case we refer to \cite{Bao Cai YY,guo,Zhangjian}.
\\
\\
 \begin{proof}[Proof of Theorem 1.1]\textbf{(i).}
 For any $u\in\mathcal{H}$ with $||u||_{L^p}=1$, by (\ref{GNinequality}) and (V) we know that, if $a \in [0,a^*)$
 \begin{eqnarray}
\nonumber E_a(u)& \geq &  \Big( 1-\frac{a}{ a^{*}}\Big)
 \int_{\mathbb{R}^{n}} |\nabla u|^p  dx+\int_{\mathbb{R}^{n}} V(x)|u(x)|^p  dx\label{adda}\\
 & \geq & \Big(  1-\frac{a}{a^{*}} \Big)\int_{\mathbb{R}^{n}} |\nabla u|^p dx \geq 0. \label{ieeeb}
 \end{eqnarray}
So,   $e(a)$ in (\ref{minimizing problem}) is well defined. Let $\{u_m \}\subset \mathcal{H}$ be a minimizing
	sequence, that is,  $||u_m||_{L^p (\mathbb{R}^n)}=1 $ and
	$\lim\limits_{m\rightarrow \infty}E_a(u_m)=e(a)$. Using (\ref{ieeeb}),
 we see that both $\int_{\mathbb{R}^{n}} |\nabla u_m(x)|^p dx  $ and $\int_{\mathbb{R}^{n}} V(x)|u_m(x)|^p dx  $
 are bounded, hence $\{u_m \}$ is bounded in $\mathcal{H}$. By  Lemma \ref{compactlemma},
 we can extract a subsequence such that
 $$ u_m \overset{m}{\rightharpoonup} u \text{ in } \mathcal{H} \text{ and } u_m \overset{m}{\rightarrow} u
 \text{ strongly in } L^q (\mathbb{R}^n  ), \text{ for any } p\leq q< p^*,$$
  for some $u\in \mathcal{H}$. Then,
 $\int_{\mathbb{R}^{n}} |u(x)|^p=1 $ and  $E_a(u)=e(a)  $ by the weak lower semi-continuity of $E_a$.
  This implies $u$ is a minimizer of $e(a)$.

 \textbf{(ii).} Choose a non-negative
 $ \varphi\in C^{\infty}_{0}(  \mathbb{R}^n )  $  such that
 $$\varphi(x)=1 \text{ if } |x|\leq 1, \text{ and } \varphi(x)=0 \text{ if } |x|\geq 2.$$
 For any $x_0\in \mathbb{R}^n, \tau >0$ and
 $R>0 $, motivated by \cite{guo} we let
 \begin{equation}\label{nomi}
 u_{\tau}(x)=A_{R,\tau}\frac{{\tau}^{\frac{n}{p}}}
 {||Q||_{L^p}   }
  \varphi \Big(\frac{x-x_0}{R}  \Big)Q \big(\tau (x-x_0) \big),
 \end{equation}
where Q is a ground state of (\ref{limit equation}) and $A_{R,\tau}$ is chosen such that $\int_{\mathbb{R}^{n}} |u_{\tau}(x)|^p dx=1  $  and then
$\lim\limits_{R\tau\rightarrow\infty} A_{R,\tau}=1$. In fact, it follows from (\ref{decay}) that
\begin{eqnarray}
\nonumber\frac{1}{A^p_{R,\tau}}&=&  \frac{\int\limits_{x\leq \tau R}Q^p(x) dx+\int\limits_{\tau R<|x|\leq 2\tau R }
{\varphi}^p (\frac{x}{\tau R})Q^p(x) dx}{||Q||_{L^p}^p}\\
&=& 1+O(e^{-\delta\tau R}), \quad as \quad \tau R\rightarrow\infty. \label{quyu1}
\end{eqnarray}
By Lemma \ref{a b Cp} and (\ref{decay}) we have
\begin{eqnarray*}
   && \int_{\mathbb{R}^{n}} \Big| \frac{1}{R}\nabla \varphi
(\frac{x}{\tau R}) Q(x)+ \tau \varphi(\frac{x}{\tau R})\nabla Q(x)  \Big|^p dx -\int_{\mathbb{R}^{n}} \Big|  \tau \varphi(\frac{x}{\tau R})\nabla Q(x)  \Big|^p dx  \\
   &\leq& \int_{\mathbb{R}^{n}} \Big|\frac{1}{R}\nabla \varphi(\frac{x}{\tau R})Q(x)\Big|^p dx + C_{p}\int_{\mathbb{R}^{n}} \Big| \frac{1}{R}\nabla \varphi(\frac{x}{\tau R}) Q(x)\Big|^{p-1}\Big| \tau \varphi(\frac{x}{\tau R})\nabla Q(x)  \Big| dx\\
   &  &  +C_{p}\int_{\mathbb{R}^{n}} \Big| \frac{1}{R}\nabla \varphi
   (\frac{x}{\tau R}) Q(x)\Big|\Big| \tau \varphi(\frac{x}{\tau R})\nabla Q(x)  \Big|^{p-1} dx\\
    &=& O(e^{-\delta\tau R}), \quad as \quad \tau R\rightarrow\infty.
\end{eqnarray*}
Then, by (\ref{quyu1}) and the exponential decay of   $Q$  (\ref{decay}), we see that
\begin{eqnarray}
\nonumber	& &\int_{\mathbb{R}^{n}}  |\nabla u_{\tau}(x)|^{p} dx-\frac{na}{n+p}\int_{{\mathbb{R}^{n}}}|u_{\tau}(x)|^{s} dx\\
\nonumber &=& \frac{A_{R,\tau}^{p} {\tau}^{n}}{ ||Q||_{L^p}^p }
\int_{\mathbb{R}^{n}} \Big|\frac{1}{R}\nabla \varphi(\frac{x-x_0}{R})Q(\tau(x-x_0)) +
\varphi(\frac{x-x_0}{R})\nabla Q(\tau(x-x_0))\tau  \Big|^{p} dx \\
\nonumber & & -\ \frac{na}{n+p}\frac{A_{R,\tau}^s{\tau}^{\frac{sn}{p}}}{||Q||_{L^{p}}^{s}}
\int_{\mathbb{R}^{n}} \Big|\varphi(\frac{x-x_0}{R}  )Q(\tau(x-x_0)) \Big|^s dx  \\
\nonumber &=&\frac{A_{R,\tau}^{p}}{||Q||_{L^{p}}^{p}}\int_{\mathbb{R}^{n}} \Big| \frac{1}{R}\nabla \varphi
(\frac{x}{\tau R}) Q(x)+ \tau \varphi(\frac{x}{\tau R})\nabla Q(x)  \Big|^p dx \\
\nonumber & & -\ \frac{na}{n+p}  \frac{A_{R,\tau}^{s}{\tau}^p}{||Q||_{L^p}^{s}  } \int_{\mathbb{R}^{n}} \Big |
\varphi (\frac{x}{R\tau})Q(x)    \Big |^{s} dx\\
\nonumber &\leq& \frac{A_{R,\tau}^{p}}{||Q||_{L^p}^p}\int_{\mathbb{R}^{n}} \Big| \tau \varphi(\frac{x}{\tau R})\nabla Q(x)
     \Big |^p dx-\frac{na}{n+p}\frac{A_{R,\tau}^{s} {\tau}^p }{||Q||_{L^p}^{s}}
     \int_{\mathbb{R}^{n}} |\varphi(\frac{x}{\tau R})Q(x)|^{s}dx \\
\nonumber    & &  +O(e^{-\delta\tau R}) \text{ as } \tau R\rightarrow\infty \\
    &\leq&\frac{\tau^{p}}   {||Q||_{L^p}^p}
 \Big(\int_{\mathbb{R}^{n}} |\nabla Q|^p dx-\frac{na}{(n+p)a^{*}}\int_{\mathbb{R}^{n}} |Q|^{s}dx\Big)+O(e^{-\delta\tau R}),
 \text{ as } R\tau \rightarrow\infty.\label{*}
\end{eqnarray}
It then follows from (\ref{*}) and (\ref{Q}) that
\begin{equation}\label{proofc}
    \int_{\mathbb{R}^{n}} |\nabla u_{\tau}|^{p} dx-\frac{na}{n+p}\int_{{\mathbb{R}^{n}}}|u_{\tau}|^{s} dx\leq\tau^{p} (1-\frac{a}{a^{*}}) +O(e^{-\delta\tau R}).
\end{equation}
On the other hand, since  $ u_{\tau}(x) $ is bounded and has compact support, the convergence
\begin{equation}
\lim\limits_{\tau \rightarrow \infty}\int_{\mathbb{R}^{n}} V(x)|u_{\tau}(x)|^p dx
= \lim\limits_{\tau\rightarrow\infty}\int_{\mathbb{R}^{n}}
\frac{V( \frac{x}{\tau} +x_0  )}{||Q||_{L^p}^p}{\varphi}^p(\frac{x}{\tau R})Q^p(x) dx =V(x_0)\label{proofd}
\end{equation}
 holds for all  $x_0\in\mathbb{R}^n$.

When $a>a^*$, it follows from (\ref{proofc}) and (\ref{proofd})
that $$ e(a)\leq \lim\limits_{\tau \rightarrow \infty}E_a(u)=-\infty. $$
This implies that for any $a>a^*$, $e(a)$ is unbounded  from below, and the
nonexistence of minimizers is therefore proved.

When $a=a^*$, taking $x_{0}\in \mathbb{R}^{n}$ such that $V(x_{0})=0$, then (\ref{proofc}) and (\ref{proofd}) imply that  $ e(a^*)\leq 0 $,
but we know that $ e(a^*)\geq 0$ by (\ref{adda}),   so $e(a^*)=0$.
If there exists a minimizer $u_0\in \mathcal{H}$ for $e(a^*)=0 $ with $||u_0||_{L^p}=1$,
 then
  \begin{equation*}
 \int_{\mathbb{R}^{n}} V(x)|u_0(x)|^p dx=\inf\limits_{x\in\mathbb{R}^n} V(x)=0,
  \end{equation*}
and
  \begin{equation*}
  \int_{\mathbb{R}^{n}} |\nabla u_0(x)|^p dx=\frac{na^*}{n+p}\int_{\mathbb{R}^{n}} |u_0(x)|^s  dx.
  \end{equation*}
These lead to a contradiction, since  the first equality implies that $u_0$ must have
 compact support, while  the second equality means that  $u_0$ has to be a nonnegative ground state  of (\ref{limit equation}) by Lemma \ref{Gagliardo-Nirenberg inequality}, thus $u_0 >0$ by the strong maximum principle \cite{maximumprinciple2}.
 So   problem (\ref{minimizing problem}) has no minimizer for $a=a^*$.

 Note that (\ref{ieeeb}) implies that $ e(a)>0 $  for $a<a^*$.
We have already shown that $e(a^*)=0 $ and  $ e(a)=-\infty$ if  $a>a^*$,
hence it remains to prove that $\lim\limits_{a\nearrow a^*}
 e(a)=0$. Indeed, let $x_0\in\mathbb{R}^n$  be such that $V(x_0)=0$, set $\tau=(a^*-a)^{-\frac{1}{p+1}}$. Then if $a\nearrow a^*$, it  follows easily from (\ref{proofc}) and (\ref{proofd}) that $ \limsup\limits_{a\nearrow a^*} e(a)\leq 0$, hence $\lim\limits_{a\nearrow a^*} e(a)=0$.
\end{proof}

\section{Blowup behavior for general trapping potential}
\noindent

In this section, we come to analyze the concentration (blowup) behavior of the ground states of (\ref{Equation}) as $a\nearrow a^*$
under the general assumption (V), that is, to give a proof of  Theorem \ref{mainresult2}.

Let  $u_{a}$    be a nonnegative minimizer of (\ref{minimizing problem}),   thus $u_{a}$ satisfies the following equation
\begin{equation}\label{Euler-Lagrange}
    -\Delta_{p}u_{a}+V(x)u_{a}^{p-1}=\mu_{a}u_{a}^{p-1}+au_{a}^{s-1},
\end{equation}
where $\mu_{a}\in \mathbb{R}$ is a suitable Lagrange multiplier.
\\
\\
\begin{proof}[Proof of Theorem \ref{mainresult2}]

\textbf{(i).} By contradiction, if (\ref{epxl}) is false, then there exists a sequence $\{a_{k}\}$ with $a_{k}\nearrow a^{*}$ as
$k\rightarrow\infty $  such that $\{u_{a_{k}}(x)\}$ is bounded in $\mathcal{H}$. By applying Lemma \ref{compactlemma}, there exist a
subsequence of $\{a_{k}\}$ (still denoted by $\{a_{k}\}$) and $u_{0}\in \mathcal{H}$ such that
\begin{equation*}
u_{a_{k}}\overset{k}{\rightharpoonup}  u_{0}  \text{ weakly in } \mathcal{H}  \text{ and }  u_{a_{k}} \xrightarrow{k} u_{0} \text{ in } L^{p}(\mathbb{R}^{n}).
\end{equation*}
Thus,
$$0=e(a^{*})\leq E_{a^{*}}(u_{0})\leq \lim \limits_{k\rightarrow\infty} E_{a_{k}}(u_{a_{k}})=
\lim \limits_{k\rightarrow\infty}e(a_{k})=0,$$
since $e(a)\rightarrow 0$ as $a\nearrow a^{*}$, by Theorem \ref{mainresult}. This  shows that $u_{0}$ is a minimizer of $e(a^{*})$,
which is impossible by  Theorem \ref{mainresult}\textbf{(ii)}. So, part \textbf{(i)} is proved.

\textbf{(ii)}.  For any solution $ u_a$ of (\ref{Euler-Lagrange}), by the result of \cite{ligongbao} we know that $u_{a}\in C_{loc}^{1,\alpha}(\mathbb{R}^{n})$ for some $ \alpha \in (0,1)$  and
$$u_{a}(x)\rightarrow 0 \quad as \quad |x|\rightarrow\infty,$$
this implies that each $u_{a}$ has at least one maximum point. Let $\bar{z}_{a}$ be a global maximum point and
define
\begin{equation}\label{w dingyi}
    \bar{w}_{a}(x)=\varepsilon^{\frac{n}{p}}_{a}u_{a}(\varepsilon_{a}x+\bar{z}_{a}),
\end{equation}
then
\begin{equation}\label{fanshu}
    \int_{\mathbb{R}^{n}}   |\nabla \bar{w}_{a}|^{p}=\int_{\mathbb{R}^{n}}   | \bar{w}_{a}|^{p}=1.
\end{equation}

By (\ref{GNinequality}), we know that
$$0\leq \int_{\mathbb{R}^{n}}|\nabla u_{a}|^{p}-\frac{n}{n+p}a\int_{\mathbb{R}^{n}}| u_{a}|^{s}=\varepsilon_{a}^{-p}
-\frac{n}{n+p}a\int_{\mathbb{R}^{n}}| u_{a}|^{s}\leq e(a).$$
By part \textbf{(i)} and Theorem \ref{mainresult} \textbf{(ii)}, we have
$$\varepsilon_{a}\rightarrow 0 \quad \text{and} \quad e(a)\rightarrow 0 \quad as \quad a\nearrow a^{*},$$
then
\begin{equation}\label{w}
    \int_{\mathbb{R}^{n}}| \bar{w}_{a}|^{s}= \varepsilon_{a}^{p}\int_{\mathbb{R}^{n}}| u_{a}|^{s}\rightarrow \frac{n+p}{na^{*}} \quad as \quad
a\nearrow a^{*}.
\end{equation}
Now, we claim that

\begin{equation}\label{yita}
\liminf \limits_{a\nearrow a^{*}} \int_{B_{2}(0)}   | \bar{w}_{a}|^{p} \geq \eta>0.
\end{equation} Indeed, it follows
from (\ref{Euler-Lagrange}) that
$$\mu_{a}=e(a)-\frac{pa}{n+p}\int_{\mathbb{R}^{n}}| u_{a}|^{s},$$
this together with (\ref{w}) indicates that
\begin{equation}\label{2}
    \varepsilon_{a}^{p}\mu_{a}\rightarrow -\frac{p}{n} \quad as \quad a\nearrow a^{*}.
\end{equation}
Moreover, by (\ref{Euler-Lagrange}) and (\ref{w dingyi}), $ \bar{w}_{a}$ satisfies that
\begin{equation}\label{Euler-Lagrange2}
    -\Delta_{p}\bar{w}_{a}(x)+\varepsilon_{a}^{p}V(\varepsilon x+\bar{z}_{a})\bar{w}_{a}^{p-1}(x)=\varepsilon_{a}^{p}\mu_{a}\bar{w}_{a}^{p-1}(x)+a\bar{w}_{a}^{s-1}(x),
\end{equation}
and since $\bar{w}_{a}\geq0$ and $ \varepsilon_{a}^{p}\mu_{a}\leq 0 $  for $a$ close to $a^*$, it follows  from (\ref{Euler-Lagrange2}) that
\begin{equation*}\label{3}
    -\Delta_{p}\bar{w}_{a}-c(x)\bar{w}_{a}^{p-1}\leq 0, \quad \text{where} \quad c(x)=a\bar{w}_{a}^{s-p}.
\end{equation*}
Thus by Theorem 7.1.1 of \cite{maximumprinciple} we have
\begin{equation}\label{4}
    \sup \limits_{B_{1}(\xi)} \bar{w}_{a}\leq C (\int_{B_{2}(\xi)} |\bar{w}_{a}|^{p})^{\frac{1}{p}},
\end{equation}
where $\xi$ is an arbitrary point in $\mathbb{R}^{n}$ and $C>0$ depends only on the upper bound of $\|c(x)\|_{L^{\frac{n}{p(1-\epsilon)}}(B_{2}(\xi))}$, i.e., $\|\bar{w}_{a}\|_{L^{\frac{p}{1-\epsilon}}(B_{2}(\xi))}$, for some $0<\epsilon\leq1$. On the other hand,
\begin{equation}\label{5}
    \bar{w}_{a}(0)\geq\zeta \quad \text{uniformly as}  \quad  a\nearrow a^{*} \quad \text{for some} \quad \zeta>0,
\end{equation}
since $0$ is a global maximum point of $ \bar{w}_{a}$. Otherwise,  there exists a sequence $a_{k}\nearrow a^{*}$ such that
$ \bar{w}_{a_{k}}(0)=\|\bar{w}_{a_{k}}\|_{L^{\infty}(\mathbb{R}^n)} \xrightarrow{k} 0$, then by concentration-compactness lemma \cite{Lions} we have
$\int_{\mathbb{R}^{n}}| \bar{w}_{a_{k}}|^{s}\rightarrow 0 $ as $k\rightarrow\infty$, which contradicts (\ref{w}).
So, (\ref{4}) and (\ref{5}) imply (\ref{yita}).

In what follows, we come to prove (\ref{dist}) by using (\ref{yita}). Since (\ref{GP}) and (\ref{GNinequality}), we have
\begin{equation*}\label{6}
    \int_{\mathbb{R}^{n}}V(x)|u_{a}|^{p}dx \leq e(a)\rightarrow 0 \quad as \quad a\nearrow a^{*},
\end{equation*}
that is
\begin{equation}\label{7}
    \int_{\mathbb{R}^{n}}V(x)|u_{a}|^{p}dx = \int_{\mathbb{R}^{n}}V(\varepsilon_{a}x+\bar{z}_{a})|\bar{w}_{a}|^{p}dx
     \rightarrow 0 \quad as \quad a\nearrow a^{*}.
\end{equation}
By contradiction, if (\ref{dist}) is false, then there is a constant $\delta>0 $ and a sequence $\{a_{k}\}$ with
$a_{k}\nearrow a^{*}$ as $k\rightarrow \infty$ such that
\begin{equation}\label{dist convergent}
   \varepsilon_{a_{k}}\xrightarrow{k} 0  \quad \text{and} \quad \lim\limits_{k\rightarrow\infty}\inf dist(\bar{z}_{a_{k}},\mathcal{A})\geq \delta >0,
\end{equation}
then,  there exists $C_{\delta}>0$ such that
\begin{equation}\label{dist convergent2}
     \liminf \limits_{k\rightarrow \infty} V(\bar{z}_{a_{k}})\geq 2C_{\delta}.
\end{equation}
Indeed, suppose such $C_{\delta}$ does not exist, then, up to a subsequence, there exist $\{\bar{z}_{a_{k}}\}\subset \mathbb{R}^n$ such that
$V(\bar{z}_{a_{k}})\overset{k}{\longrightarrow} 0$. Since $V(x)$ satisfies (V), we have $\{\bar{z}_{a_{k}}\}$ is bounded and thus $\bar{z}_{a_{k}}\overset{k}{\longrightarrow} z_0 $, for some $z_0 \in \mathbb{R}^n$. By the continuity of $V(x)$ we have $z_0 \in \mathcal{A}$,
but this contradicts (\ref{dist convergent}), so (\ref{dist convergent2}) is proved.
By Fatou's Lemma and (\ref{yita}), we see that
$$ \lim\limits_{k\rightarrow \infty}\int_{\mathbb{R}^{n}}V(\varepsilon_{a_{k}}x+\bar{z}_{a_{k}})|\bar{w}_{a_{k}}|^{p}dx
\geq\int_{\mathbb{R}^{n}}\lim\limits_{k\rightarrow \infty}V(\varepsilon_{a_{k}}x+\bar{z}_{a_{k}})|\bar{w}_{a_{k}}|^{p}dx \geq C_{\delta}\eta,$$
which contradicts (\ref{7}). Therefore, (\ref{dist}) holds.

\textbf{(iii).}
Let $\{a_{k}\}$ be a sequence such that $a_{k}\nearrow a^{*}$ as $k\rightarrow\infty$.
For simplicity, we set
$$u_{k}(x):=u_{a_{k}}(x), \quad \bar{w}_{k}:=\bar{w}_{a_{k}}\geq 0, \quad \bar{z}_{k}:=\bar{z}_{a_{k}} \text{  and }\quad \varepsilon_{k}:=\varepsilon_{a_{k}}>0.$$
By (\ref{dist}), (\ref{fanshu}) and (\ref{w}),  there exists a subsequence of $\{a_{k}\}$, still denoted by $\{a_{k}\}$, such that
$$\lim\limits_{k\rightarrow\infty} \bar{z}_{k}=x_{0} \quad \text{with} \quad V(x_{0})=0,$$
and
$$ \bar{w}_{k} \overset{k}{\rightharpoonup} w_{0}\geq0 \quad \text{weakly in} \quad  W^{1,p}(\mathbb{R}^{n})$$
for some $w_{0}\in W^{1,p}(\mathbb{R}^{n})$. Moreover, $\bar{w}_{k}$ satisfies
\begin{equation*}\label{8}
     -\Delta_{p}\bar{w}_{k}(x)+\varepsilon_{k}^{p}V(\varepsilon_{k} x+\bar{z}_{k})\bar{w}_{k}^{p-1}(x)=\varepsilon_{k}^{p}\mu_{k}\bar{w}_{k}^{p-1}(x)+a_{k}\bar{w}_{k}^{s-1}(x).
\end{equation*}
%%and note from (\ref{2}) that $\varepsilon_{k}^{p}\mu_{k}\xrightarrow{k} -\frac{p}{n}$.

Motivated by the idea of \cite{yangjianfu zhuxiping,ligongbao Acta}, we claim that
\begin{equation}\label{gai7}
    \bar{w}_{k} \overset{k}{\rightarrow} w_{0} \quad \text{strongly in} \quad  W^{1,p}(\mathbb{R}^{n}).
\end{equation}
 Indeed, by Theorem \ref{mainresult} and the definition of $\bar{w}_{k}$ we have
\begin{equation}\label{gai1}
   0\leq \lim\limits_{k\rightarrow\infty} \Big ( \varepsilon_{k}^{-p}\int_{\mathbb{R}^n}| \nabla  \bar{w}_{k}|^p
    -\frac{na_{k}}{n+p}\varepsilon_{k}^{-p}\int_{\mathbb{R}^n}|\bar{w}_{k}|^s \Big )
    \leq \lim\limits_{k\rightarrow\infty} e(a_k)=0,
\end{equation}
and thus by (\ref{fanshu}) and (\ref{gai1}) that
\begin{equation*}
    \int_{\mathbb{R}^n}|\bar{w}_{k}|^p\equiv 1,
\end{equation*}
and
\begin{equation}\label{gai2}
\lim\limits_{k\rightarrow\infty} \Big ( \int_{\mathbb{R}^n}| \nabla  \bar{w}_{k}|^p
    -\frac{na^*}{n+p}\int_{\mathbb{R}^n}|\bar{w}_{k}|^s \Big )=0.
\end{equation}
To prove (\ref{gai7}), we show first that
\begin{equation}\label{gai3}
    \bar{w}_{k} \overset{k}{\rightarrow} w_{0} \quad \text{strongly in} \quad  L^p(\mathbb{R}^{n}).
\end{equation}
By Lemma III.1 of \cite{Lions}, to show (\ref{gai3}) we only need to exclude the ``vanishing case'' and ``dichotomy case'' of  the function sequence $\{ |\bar{w}_{k}|^p \}$.
In fact, we can easily rule out the ``vanishing case'' by (\ref{yita}). Then, arguing indirectly, suppose the ``dichotomy case'' occurs and taking (\ref{yita}) into consideration,
 there exists   $0<\rho<1 $ such that for any $\varepsilon>0$ there are  $R_k>0$, $\{y_k\}\subset \mathbb{R}^n$  and two function sequences
 $ \{w^1_k \},\{w^2_k \}\subset W^{1,p}(\mathbb{R}^{n})$ with $\text{Supp } w^1_k \subset B_{R_k}(y_k) $ and $\text{Supp } w^2_k \subset  \mathbb{R}^{n}\setminus B_{2R_k}(y_k) $
satisfiy
\begin{equation}\label{gai8}
    \Big|\int \limits_{B_{R_k}(y_k)} |w^1_k|^p -\rho\Big|\leq\varepsilon, \quad
    \Big|\int \limits_{\mathbb{R}^{n}\setminus B_{2R_k}(y_k)} |w^2_k|^p-( 1-  \rho)\Big| \leq\varepsilon;
\end{equation}
and, up to a subsequence,
\begin{equation}\label{gai11}
    \int_{\mathbb{R}^n} |\bar{w}_k-(w^1_k+w^2_k)|^s\leq C_{\varepsilon} \overset{\varepsilon\rightarrow 0}{\longrightarrow} 0,
\end{equation}
\begin{equation}\label{gai9}
\liminf \limits_{k\rightarrow\infty}   \int_{\mathbb{R}^n} \Big( |\nabla \bar{w}_k|^p - |\nabla w^1_k|^p-|\nabla w^2_k|^p \Big)\geq 0.
\end{equation}
Then, by (\ref{gai8})-(\ref{gai9}), (\ref{GNinequality})  and (\ref{w})
%%H\"older inequality
we have
\begin{eqnarray}\label{gai10}
 \nonumber  &&   \Big(\int_{\mathbb{R}^n} |\nabla \bar{w}_k|^p -\frac{na^*}{n+p}\int_{\mathbb{R}^n}|\bar{w}_{k}|^s\Big) \\
  %\nonumber &\geq&\int \limits_{B_{R}(0)} \Big(|\nabla \bar{w}_k|^p -\frac{na^*}{n+p}|\bar{w}_{k}|^s\Big)+\int \limits_{\mathbb{R}^n\setminus B_{R}(0)} \Big(|\nabla \bar{w}_k|^p -\frac{na^*}{n+p}|\bar{w}_{k}|^s\Big) \\
  \nonumber &\geq&\int \limits_{B_{R_k}(y_k)}\Big( |\nabla w^1_k|^p -\frac{na^*}{n+p}|w^1_{k}|^s\Big)+\int \limits_{\mathbb{R}^n\setminus B_{2R_k}(y_k)}\Big( |\nabla w^2_k|^p -\frac{na^*}{n+p}|w^2_{k}|^s\Big)-C_{\varepsilon}\\
  \nonumber &\geq& \frac{na^*}{n+p}\Big(  \big( (\rho+\varepsilon)^{-\frac{p}{n}}-1    \big)\int_{\mathbb{R}^n}|w^1_{k}|^s   +     \big( (1-\rho+\varepsilon)^{-\frac{p}{n}}-1    \big)\int_{\mathbb{R}^n}|w^2_{k}|^s  \Big) -C_{\varepsilon}\\
 &>& 0,
\end{eqnarray}
since  $\varepsilon$ can be arbitrarily small.
Then (\ref{gai10})  leads to a contradiction with (\ref{gai2}). So,
(\ref{gai3}) holds. (\ref{gai3}) and the boundedness of  $\{\bar{w}_{k}\}$ in $W^{1,p}(\mathbb{R}^{n}) $ imply that
\begin{equation}\label{gai4}
    \bar{w}_{k} \overset{k}{\rightarrow} w_{0} \quad \text{strongly in} \quad  L^s(\mathbb{R}^{n}).
\end{equation}
Combine (\ref{gai3})(\ref{gai4}) with (\ref{GNinequality}) and (\ref{w}), we have $\int_{\mathbb{R}^n}| \nabla  w_{0}|^p \geq 1$ and thus
\begin{equation*}\label{gai5}
    \int_{\mathbb{R}^n}| \nabla  w_{0}|^p = 1,
\end{equation*}
by the weak lower semi-continuity of the norm.
Since $W^{1,p}(\mathbb{R}^{n})$ is a uniformly convex  Banach space \cite{Adams}, (\ref{gai7}) holds by \cite{Brezis}.
%Thus $w_0$ is also an extremal of (\ref{GNinequality}).
Moreover,    $w_0$ satisfies
\begin{equation*}
    -\Delta_{p} w_0+\frac{p}{n}{w_0}^{p-1}=a^{*} {w_0}^{s-1},
\end{equation*}
and

\begin{equation*}
    \int_{\mathbb{R}^n}| \nabla  w_{0}|^p = \int_{\mathbb{R}^n}|  w_{0}|^p =1.
\end{equation*}
So, %Lemma \ref{Gagliardo-Nirenberg inequality}   and
Lemma \ref{lemma 2.2} implies that
\begin{equation}\label{gai6}
   w_{0}=\frac{Q(x)}{{a^*}^{\frac{n}{p^2}}},
\end{equation}
for some $Q \in \mathcal{G}$.
Then (\ref{gai7}) and (\ref{gai6}) gives the proof of Theorem \ref{mainresult2}.
\end{proof}

\section{Refined blowup behavior for polynomial type potential}
\noindent

This section is devoted to prove Theorem \ref{mainresult3}. In what follows,  we always denote $\{a_k\}$ to be the convergent subsequence
in Theorem \ref{mainresult2}. Our first lemma is to address the upper bound of $e(a_k)$.
\begin{lemma}\label{lemma}
If $V(x)$ satisfies  (\ref{polynomial V}), then
\begin{equation}\label{11}
    0\leq e(a_k)\leq \frac{(a^*-a_k)^{\frac{q}{p+q}}}{{a^{*}}^{\frac{n+q}{p+q}}}\lambda^{\frac{p}{p+q}} \big((\frac{q}{p})^{\frac{p}{p+q}}+(\frac{p}{q})^{\frac{q}{p+q}}+o(1)\big) \quad \text{as} \quad k\rightarrow\infty,
\end{equation}
where $\lambda$ is given by (\ref{lambda}) and $o(1)$ is a quantity  depends only on $k$.

\end{lemma}
\begin{proof}
Let  $Q \in \mathcal{G}$  be  given in Theorem \ref{mainresult2} and $y_0$ as in (\ref{y_0}). Take
\begin{equation*}
     u(x)=A_{R,\tau}\frac{{\tau}^{\frac{n}{p}}}
 {||Q||_{L^p}   }
  \varphi \Big(\frac{x-x_0}{R}  \Big)Q(\tau (x-x_0)-y_0)
\end{equation*}
with $x_{0}\in \mathcal{Z}$ as  a trial function. Let  $R=\frac{1}{\sqrt{\tau}}$ and it is easy to see $R\rightarrow0$ and $\tau R\rightarrow \infty$ as
$\tau\rightarrow\infty$. Then,
\begin{eqnarray}
\nonumber\int_{\mathbb{R}^{n}}V(x)|u|^{p}&=& A_{R,\tau}^{p} \int_{\mathbb{R}^{n}}V(x)\frac{\tau^{n}}{\|Q\|_{L^{p}}^{p}} \varphi^{p} (\frac{x-x_0}{R} )Q^{p}(\tau (x-x_0)-y_0) dx \\\nonumber
   &=&\frac{A_{R,\tau}^{p}}{\|Q\|_{L^{p}}^{p}} \int_{\mathbb{R}^{n}} \tau^{n} |x-x_{0}|^{q}\varphi^{p} (\frac{x-x_0}{R} )Q^{p}(\tau (x-x_0)-y_0)\frac{V(x)}{|x-x_{0}|^{q}} dx \\\nonumber
   &\leq & \frac{A_{R,\tau}^{p}}{\|Q\|_{L^{p}}^{p}} \tau^{-q} \Big( \int_{\mathbb{R}^{n}} |x|^{q}Q^{p}(x-y_0)dx \lim \limits_{x\rightarrow x_{0}} \frac{V(x)}{|x-x_{0}|^{q}}+o(1)\Big) \\
   &\leq& A_{R,\tau}^{p} {a^{*}}^{-\frac{n}{p}} \tau^{-q}\big( \lambda+o(1)\big),\label{12}
\end{eqnarray}
with $o(1)\rightarrow 0 $ as $\tau\rightarrow\infty $. By (\ref{12}), (\ref{quyu1}) and (\ref{proofc}),
we see that, for large $\tau$,
\begin{eqnarray*}
    e(a_k) &\leq & E_{a_k}(u)=\int_{\mathbb{R}^{n}} |\nabla u|^{p} - \frac{n}{n+p}a_k\int_{\mathbb{R}^{n}}|u|^{s}+ \int_{\mathbb{R}^{n}}V(x)|u|^{p}\\
   &\leq&  \frac{a^{*}-a_k}{a^{*}} \tau^{p}+ {a^{*}}^{-\frac{n}{p}} \tau^{-q} \big(\lambda+o(1)\big)+O(e^{-\delta R\tau}).
\end{eqnarray*}
Choose $\tau=(\frac{{a^{*}}^{1-\frac{n}{p}}\lambda q}{(a^{*}-a_k)p})^{\frac{1}{p+q}}$ and thus $\tau \overset{k}{\rightarrow}\infty$, then
\begin{equation*}
     0\leq e(a_k)\leq \frac{(a^*-a_k)^{\frac{q}{p+q}}}{{a^{*}}^{\frac{n+q}{p+q}}}\lambda^{\frac{p}{p+q}} \big((\frac{q}{p})^{\frac{p}{p+q}}+(\frac{p}{q})^{\frac{q}{p+q}}+o(1)\big) \quad \text{as} \quad k\rightarrow\infty.
\end{equation*}
\end{proof}

Based on Lemma \ref{lemma} and Theorem \ref{mainresult2}, we are ready to prove Theorem \ref{mainresult3}.
\begin{proof}[Proof of Theorem \ref{mainresult3}]

 \textbf{(i).}
From (\ref{GNinequality}) we have
\begin{eqnarray*}
  e(a_k)=E_{a_k}(u_{a_{k}}) &=&  \int_{\mathbb{R}^{n}} |\nabla u_{a_k}|^{p} - \frac{n}{n+p}a\int_{\mathbb{R}^{n}}|u_{a_k}|^{s}+ \int_{\mathbb{R}^{n}}V(x)|u_{a_k}|^{p}\\
   &\geq& \frac{a^{*}-a_k}{a^{*}}\int_{\mathbb{R}^{n}} |\nabla u_{a_k}|^{p}+\int_{\mathbb{R}^{n}}V(x)|u_{a_k}|^{p} \\
   &=& \frac{a^{*}-a_k}{a^{*}} \varepsilon_{a_k}^{-p} + \int_{\mathbb{R}^{n}}V(\varepsilon_{a_k}x+\bar{z}_{a_k})|\bar{w}_{a_k}|^{p} dx,
\end{eqnarray*}
where $\bar{w}_{a_k}$ is given by (\ref{w dingyi}).
By Theorem\ref{mainresult2}, we may assume that $\bar{z}_{a_k}\rightarrow x_{i}$ with $x_{i}\in\mathcal{A}$ as $k\rightarrow\infty$. Define
$$\bar{V}_{a_{k}}(x)=\frac{V(\varepsilon_{a_{k}}x+\bar{z}_{a_k})}{c_{i}|\varepsilon_{a_k}x+\bar{z}_{a_k}-x_{i}|^{q_{i}}}
\quad \text{with} \quad
c_{i}=\lim \limits_{x\rightarrow x_{i}}\frac{V(x)}{|x-x_{i}|^{q_{i}}},$$
then $\bar{V}_{a_k}(x)\rightarrow 1$  a.e. in $x\in \mathbb{R}^n$ as $k\rightarrow\infty$.
Hence,
\begin{eqnarray*}
  \int_{\mathbb{R}^{n}}V(\varepsilon_{a_{k}}x+\bar{z}_{a_{k}})|\bar{w}_{a_{k}}|^{p} dx &=& \int_{\mathbb{R}^{n}}\bar{V}_{a_{k}}c_{i}|\varepsilon_{a_{k}}x+\bar{z}_{a_{k}}-x_{i}|^{q_{i}}\bar{w}_{a_{k}}^{p} dx \\
   &=& \varepsilon_{a_{k}}^{q_{i}} \int_{\mathbb{R}^{n}}\bar{V}_{a_{k}}c_{i}|x+\frac{\bar{z}_{a_{k}}-x_{i}}{\varepsilon_{a_{k}}}|^{q_{i}}\bar{w}_{a_{k}}^{p} dx.
\end{eqnarray*}
We claim that $\limsup\limits_{k\rightarrow\infty}\frac{|\bar{z}_{a_k}-x_{i}|}{\varepsilon_{a_{k}}}<\infty$ and $q_{i}=q$. Indeed, let
\begin{equation*}
    \rho_{a_k}:=\int_{\mathbb{R}^{n}}\bar{V}_{a_k}c_{i}|x+\frac{\bar{z}_{a_k}-x_{i}}{\varepsilon_{a_k}}|^{q_{i}}\bar{w}_{a_k}^{p} dx,
\end{equation*}
then $\rho_{a_k}\rightarrow\infty$ if $\frac{|\bar{z}_{a_k}-x_{i}|}{\varepsilon_{a_k}}\rightarrow\infty$. Moreover we have
\begin{eqnarray}
 \nonumber e(a_k) &\geq& \frac{a^{*}-a_k}{a^{*}} \varepsilon_{a_k}^{-p}+ \rho_{a_k}\varepsilon_{a_k}^{q_{i}}\\
   &\geq&  (\frac{a^{*}-a_k}{a^{*}})^{\frac{q_{i}}{p+q_{i}}} \rho_{a_k}^{\frac{q_{i}}{p+q_{i}}}  ((\frac{q}{p})^{\frac{p}{p+q}}+(\frac{p}{q})^{\frac{q}{p+q}}), \label{xing}
\end{eqnarray}
which contradicts Lemma \ref{lemma} if $\frac{\bar{z}_{a_k}-x_{i}}{\varepsilon_{a_k}}\xrightarrow{k}\infty$. Thus
\begin{equation*}
    \limsup \limits_{k\rightarrow\infty}\frac{|\bar{z}_{a_k}-x_{i}|}{\varepsilon_{a_k}} <\infty.
\end{equation*}
By Fatou's Lemma and the above fact,  we see that
\begin{equation*}
    \liminf \limits_{k\rightarrow\infty} \rho_{a_k} >0,
\end{equation*}
the above inequality together with (\ref{xing}) and Lemma \ref{lemma} indicates that $q_i=q$. Then using  Fatou's Lemma again,
\begin{eqnarray}
 \nonumber && \liminf \limits_{k\rightarrow\infty} \int_{\mathbb{R}^{n}}\bar{V}_{a_k}|x+\frac{\bar{z}_{a_k}-x_{i}}{\varepsilon_{a_k}}|^{q}w_{a_k}^{p} dx  \lim \limits_{x\rightarrow x_{i}}\frac{V(x)}{|x-x_{i}|^{q}}\\
   &\geq& \frac{1}{{a^{*}}^{\frac{n}{p}}}
    \int_{\mathbb{R}^{n}}|x+y_0|^{q} Q^{p} dx  \lim \limits_{x\rightarrow x_{i}}\frac{V(x)}{|x-x_{i}|^{q}}
   \geq  \frac{1}{{a^{*}}^{\frac{n}{p}}} \lambda, \label{houmianyong}
\end{eqnarray}
where  (\ref{y_0}) and (\ref{lambda}) were used in the last two inequalities.
So,
\begin{eqnarray}
 \nonumber e(a_k) &\geq&  \frac{a^{*}-a_k}{a^{*}} \varepsilon_{a_k}^{-p} + \varepsilon_{a_k}^{q}\frac{1}{{a^{*}}^{\frac{n}{p}}}\lambda \\
  &\geq& \frac{(a^*-a_k)^{\frac{q}{p+q}}}{{a^{*}}^{\frac{n+q}{p+q}}}\lambda^{\frac{p}{p+q}} ((\frac{q}{p})^{\frac{p}{p+q}}+(\frac{p}{q})^{\frac{q}{p+q}}),\label{15}
\end{eqnarray}
where the equality in the last inequality holds if and only if $$\varepsilon_{a_k}={a^{*}}^{\frac{n-p}{p(p+q)}}
 (a^{*}-a_k)^{\frac{1}{p+q}} \lambda^{-\frac{1}{p+q}}(\frac{p}{q})^{\frac{1}{p+q}}. $$
This together with Lemma \ref{lemma}  shows Theorem \ref{mainresult3}\textbf{(i)}.

\textbf{(ii).} By Lemma \ref{lemma} we know that the inequality (\ref{houmianyong}) is in fact an equality which implies $x_{0}\in \mathcal{Z}$. Then we need only to show that
\begin{equation} \label{16}
    \lim \limits_{k\rightarrow\infty }\frac{\varepsilon_{k}}{\sigma_{k}}=1.
\end{equation}
If (\ref{16})  is false, then there exists a subsequence of $\{k\}$, still denoted by  $\{k\}$,  such that
\begin{equation*}
    \lim \limits_{k\rightarrow\infty }\frac{\varepsilon_{k}}{\sigma_{k}}=\theta \neq1, \quad \text{with} \quad 0\leq\theta\leq\infty.
\end{equation*}
From (\ref{15}) we have
\begin{eqnarray*}
  e(a_k) &\geq&  \frac{a^{*}-a_{k}}{a^{*}} \varepsilon_{k}^{-p} + \varepsilon_{k}^{q}\frac{1}{{a^{*}}^{\frac{n}{p}}}\lambda \\
  \nonumber &\geq& \frac{(a^*-a_k)^{\frac{q}{p+q}}}{{a^{*}}^{\frac{n+q}{p+q}}}\lambda^{\frac{p}{p+q}} \big((\frac{q}{p})^{\frac{p}{p+q}}\theta^{-p}+(\frac{p}{q})^{\frac{q}{p+q}}\theta^{q}\big)\\
  &>& \frac{(a^*-a_k)^{\frac{q}{p+q}}}{{a^{*}}^{\frac{n+q}{p+q}}}\lambda^{\frac{p}{p+q}} \big((\frac{q}{p})^{\frac{p}{p+q}}+(\frac{p}{q})^{\frac{q}{p+q}}\big),
\end{eqnarray*}
when $\lim \limits_{k\rightarrow\infty }\frac{\varepsilon_{k}}{\sigma_{k}}=\theta \neq1$,  and this contradicts Lemma \ref{lemma}. (\ref{16}) together with
(\ref{jixian}) gives (\ref{jixian2}) and (\ref{sigma}). So the proof is complete.

\end{proof}

{\bf Declaration of interest:} We confirm that there are no conflicts of interest associated with this publication.

{\bf Acknowledgements:} This work was supported by the NFSC [grant numbers 11471331,11501555].

\end{document}